\documentclass[a4paper,12pt]{article}
\usepackage{amssymb,amsmath,amsthm,latexsym}
\usepackage[pdftex,bookmarks,colorlinks=false]{hyperref}
\usepackage[hmargin=1.2in,vmargin=1.2in]{geometry}

\newtheorem{theorem}{Theorem}[section]

\newtheorem{corollary}[theorem] {Corollary}
\newtheorem{definition}[theorem]{Definition}

\newtheorem{proposition}[theorem]{Proposition}
\newtheorem{remark}[theorem]{Remark}

\title{\textbf{\sc On Weak Integer Additive Set-Indexers of Certain Graph Classes}}

\author{{\bf N K Sudev} $^{{1},{\ast}}$ and {\bf K A Germina $^{2}$}
\\ \\
$^{1}${\small Department of Mathematics}\\ {\small Vidya Academy of Science \& Technology} \\ {\small  Thalakkottukara, Thrissur - 680501, Kerala, India.}\\ {\small email: {\em sudevnk@gmail.com}}
\\ \vspace{0.3cm}
$^{\ast}$ {\small Corresponding author.}
\\
$^{2}${\small Department of Mathematics} \\ {\small School of Mathematical \& Physical Sciences} \\ {\small Central University of Kerala, Kasaragod - 671316, Kerala, India.}\\ {\small email: {\em srgerminaka@gmail.com}}
}
\date{}
\begin{document}
\maketitle

\begin{abstract}
Let $\mathbb{N}_0$ denote the set of all non-negative integers and $\mathcal{P}(\mathbb{N}_0)$ be its power set. An integer additive set-indexer (IASI) of a graph $G$ is an injective function $f:V(G)\to \mathcal{P}(\mathbb{N}_0)$ such that the induced function $f^+:E(G) \to \mathcal{P}(\mathbb{N}_0)$ defined by $f^+ (uv) = f(u)+ f(v)$ is also injective. An IASI $f$ is said to be a weak IASI if $|f^+(uv)|=\max(|f(u)|,|f(v)|)$ for all adjacent vertices $u,v\in V(G)$. The sparing number of a weak IASI graph $G$ is the minimum number of edges in $G$ with singleton set-labels. In this paper, we study the admissibility of weak integer additive set-indexers by certain graph classes and associated graphs of given graphs.
\end{abstract}


\section{Introduction}

For all  terms and definitions, not defined specifically in this paper, we refer to \cite{FH} and \cite{BLS}. Unless mentioned otherwise, all graphs considered here are simple, finite and have no isolated vertices.

The {\em sum set} of two sets $A, B$, denoted by  $A+B$, is defined by $A + B = \{a+b: a \in A, b \in B\}$. Using the concepts of sum sets, the notion of integer additive set-indexers is introduced in \cite{GA} as follows. 

Let $\mathbb{N}_0$ denote the set of all non-negative integers. An {\em integer additive set-indexer} (IASI) is an injective function $f:V(G)\to \mathcal{P}(\mathbb{N}_0)$ such that the induced function $f^+:E(G) \to \mathcal{P}(\mathbb{N}_0)$ defined by $f^+ (uv) = f(u)+ f(v)$ is also injective. An IASI is said to be $k$-uniform if $|f^+(uv)|=k$ for all $u,v\in V(G)$.

The cardinality of the labeling set of an element (vertex or edge) of a graph $G$ is called the {\em set-indexing number} of that element. An element (a vertex or an edge) of graph which has the set-indexing number $1$ is called a {\em mono-indexed element} of that graph.

\noindent The notion of a weak IASI was introduced in \cite{GS3} as follows. 

A {\em weak IASI} is an IASI $f$ such that $|f^+(uv)|=\max(|f(u)|,|f(v)|)$ for all $u,v\in V(G)$. A graph which admits a weak IASI is known as a weak IASI graph. In a weak IASI graph $G$, at least one end vertex of every edge of $G$ has the set-indexing number $1$.

The {\em sparing number} of a graph $G$ is defined to be the minimum number of mono-indexed edges required for $G$ to admit a weak IASI and is denoted by $\varphi(G)$.

\noindent The main results on weak IASI graphs are the following. 

\begin{theorem}\label{T-WUOC}
\cite{GS3} An odd cycle $C_n$ has a weak IASI if and only if it has at least one mono-indexed edge. That is, the sparing number of an odd cycle is $1$. 
\end{theorem}

\begin{theorem}\label{T-MIEOE}
\cite{GS3} An odd cycle $C_n$ that admits a weak IASI has odd number mono-indexed edges and an even cycle $C_n$ that admits a weak IASI has even number mono-indexed edges.
\end{theorem}

\begin{theorem}\label{T-WUOC1}
\cite{GS3} A bipartite graph need not contain any mono-indexed edges. That is, the sparing number of bipartite graphs is $0$.
\end{theorem}

\begin{theorem}\label{T-WKN}
\cite{GS3} A complete graph $K_n$ admits a weak IASI if and only if the maximum number of mono-indexed edges of $K_n$ is $\frac{1}{2}(n-1)(n-2)$.
\end{theorem}

\begin{theorem}\label{T-WIUG}
\cite{GS4} The union $G_1\cup G_2$ of two weak IASI graphs $G_1$ and $G_2$ admits a weak IASI. Moreover, $\varphi(G_1\cup G_2)=\varphi(G_1)+\varphi(G_2)-\varphi(G_1\cap G_2)$.
\end{theorem}

In this paper, we discuss first on the admissibility of weak IASI by certain graphs which contain cliques or independent sets or both.

\section{New Results}

In view of Theorem \ref{T-WKN}, we note that a complete graph can have at most one mono-indexed vertex. Therefore, we propose

\begin{proposition}\label{C-WKN}
The sparing number of a complete graph $K_n$ is equal to the number of triangles which contain the vertex that is not mono-indexed in $K_n$.
\end{proposition}
\begin{proof}
Every pair among the $n-1$ edges incident on every vertex $v$ of a complete graph $K_n$ form a triangle with one edge of $K_n$ which does not incident on $v$. Hence, the number of triangles containing $v$ is $\binom{n-1}{2}=\frac{1}{2}(n-1)(n-2)$. Therefore, by Theorem \ref{T-WKN}, the number of triangles containing the single vertex $v$, that is not mono-indexed, is $\varphi(K_n)$. 
\end{proof}

In the following discussion, we study the sparing number of certain graph one of whose components is a complete graph. Now, recall the definition of a sun graph.

\begin{definition}{\rm
\cite{BLS} An {\em $n$-sun} or a {\em trampoline}, denoted by $S_n$,  is a chordal graph on $2n$ vertices, where $n\ge 3$, whose vertex set can be partitioned into two sets $U = \{u_1,u_2,u_3,\ldots, u_n\}$ and $W = \{w_1,w_2,w_3,\ldots, w_n\}$ such that $W$ is an independent set of $G$ and $w_j$ is adjacent to $u_i$ if and only if $j=i$ or $j=i+1~(mod ~ n)$. {\em A complete sun} is a sun $G$ where the induced subgraph $\langle U \rangle$ is complete.} 
\end{definition}

\begin{theorem}\label{T-SNKSG1}
The sparing number of a complete sun graph $S_{n}$ is $\frac{1}{2}(n^2-3n+6)$.
\end{theorem}
\begin{proof}
Let $U$ and $W$ be the partitions of $V(S_n)$, where $W$ is an independence set. Since $S_n$ is a complete sun, $\langle U \rangle=K_n$. Let $f$ be a weak IASI on $S_n$. 

If we first label the vertices of $U$, then,  since $ \langle U \rangle =K_n$, by Theorem \ref{T-WKN}, exactly one vertex, say $u$, in $U$ can have a non-singleton set-label and the sparing number of $\langle U \rangle$ is $\frac{1}{2}(n-1)(n-2)$. Then, two vertices in $W$ that are adjacent to the vertex $u\in U$ must be mono-indexed and the other two edges incident on these vertices of $W$ (but not on $u$) are also mono-indexed. Therefore, the number of mono-indexed edges in $S_n$, in this case, is $\frac{1}{2}(n-1)(n-2)+2 = \frac{1}{2}(n^2-3n+6)$. 

If we label the vertices of $W$ first, then, since $W$ is an independent set, we can label all the vertices of $W$ by distinct non-singleton sets. Therefore, all the vertices of the component $\langle U \rangle$ must be mono-indexed. In this case, the total number of mono-indexed edges in $S_n$, in this case, is $\frac{1}{2}n(n-1)$. 

For any positive integer $n\ge 3$, we have $n^2-3n+6 \le n^2-n$.  Hence, we have $\varphi(S_n)=\frac{1}{2}n(n-1)$.
\end{proof}

\noindent Let us now consider the class of  split graphs, defined as follows.

\begin{definition}{\rm
\cite{BLS} A {\em split graph} is a graph in which the vertices can be partitioned into a clique $K_r$ and an independent set $S$. A split graph is said to be a {\em complete split graph} if every vertex of the independent set $S$ is adjacent to every vertex of the the clique $K_r$ and is denoted by $K_S(r,s)$, where $r$ and $s$ are the orders of $K_r$ and $S$ respectively. }
\end{definition}

\noindent The following result discusses the sparing number of a split graph.

\begin{theorem}\label{T-SNSG}
The sparing number of a split graph $G$ $G$ is equal to the number of triangles in $G$ containing the vertex that is not mono-indexed in its clique. 
\end{theorem}
\begin{proof}
Let $K_r$ be the clique and $S$ be the independent set in the split graph $G$. Let $\{u_1,u_2,\ldots,u_r\}$ be the vertex set of $K_r$ and $S=\{v_1,v_2,\ldots, v_l\}$.  By Theorem \ref{T-WKN}, $K_r$ can have at most one  vertex that is not mono-indexed and its sparing number is $\frac{1}{2}(r-1)(r-2)$. Choose the vertex $u$ of $K_r$ which is contained in minimum number of triangles with one vertex in $S$,  to label with a non-singleton set. Now let $\mathit{\eta}$ be the number of triangles that contain the vertex $u$. If there exists a triangle incident on $u_i$ with one vertex $v_j$ in $S$, then the edge $v_ju_k$ of the triangle $u_iv_ju_k$ must be mono-indexed. By Corollary \ref{C-WKN}, the number of triangles incident on $u$ in $K_r$ is $^{r-1}C_2=\varphi(K_r)$. Therefore, the number of triangles incident on $u$ with one vertex in $S$ is $\mathit{\eta}-\frac{(r-1)(r-2)}{2}$. Hence, the number of mono-indexed edges in $G$ is  $\varphi(G)=\frac{(r-1)(r-2)}{2}+\mathit{\eta}-\frac{(r-1)(r-2)}{2} = \mathit{\eta}$. This completes the proof.
\end{proof}

The following theorem estimates the sparing number of a complete split graph $G$.

\begin{theorem}\label{T-SNKSG}
The sparing number of a complete split graph is equal to the sparing number of the maximal clique in it.
\end{theorem}
\begin{proof}
Let $U=\{u_1,u_2,u_3,\ldots,u_r\}$ and $S=\{w_1,w_2, w_3,\ldots,w_s\}$, where $\langle U\rangle =K_r$ and $W$ is an independent set in $G$. 

If the vertices of $U$ are labeled first, then by Theorem \ref{T-WKN}, exactly one vertex of $U$, say $u$, can have a non-singleton set-label and the number of mono-indexed edges in $\langle U \rangle =K_r$ is $\frac{1}{2}(r-1)(r-2)$. Since every vertex of $S$ is adjacent to all vertices of $U$, each vertex $w_j, 1\le j \le s$, in $S$ must be mono-indexed. Therefore, all the edges incident on $w_j$, except the edge $uw_j$ are mono-indexed. That is, there exist exactly $n-1$ mono-indexed edges in $G$ corresponding to each vertex in $S$. Hence, the number of mono-indexed edges between $U$ and $S$ is $s(r-1)$. Therefore, the total number of mono-indexed edges in $G$ is $\frac{1}{2}(r-1)(r-2)+s(r-1)=\frac{1}{2}(r-1)(r+2s)$.

If the vertices of $S$ are labeled first, all vertices of $S$ can be labeled by distinct non-singleton sets, since $S$ is an independent set of $G$. Hence, every vertex of $U$ must be mono-indexed. Hence, the number of mono-indexed edges in $\langle U \rangle$ is $\frac{r(r-1)}{2}$ there is no mono-indexed edges between $U$ and $S$. Hence, the total number of mono-indexed edges in this case is $\frac{r(r-1)}{2}$. 

For any positive integer $s, r<r+2s$ and hence $\frac{1}{2}r(r-1)<\frac{1}{2}(r-1)(r+2s)$. Hence, $\varphi(K_S(r,s))=\frac{1}{2}r(r-1)$. This completes the proof.
\end{proof}

\noindent Next, consider the definition of a bisplit graph.

\begin{definition}{\rm
\cite{BHLL} An undirected graph $G$ is a {\em bisplit graph} if its vertex set $V$ can be partitioned into three independent sets $X,Y {\text{and}}~ Z$ such that $Y\cup Z$ induces a complete bipartite subgraph (a bi-clique) in $G$.}
\end{definition}

\begin{proposition}\label{T-SNBSG1}
Let $G$ be a bisplit graph and let $X, Y, Z$  be the three partitions of $V(G)$. Then, the sparing number of $G$ is the number of paths of length $2$ with internal vertex in the set with the least cardinality.
\end{proposition}
\begin{proof}
Without loss of generality, assume that $|X|\le |Y|\le |Z|$. Then, label all the vertices of $Z$ by distinct non-singleton sets  and label the vertices of $Y$ by distinct singleton sets. Then, no edges between $Y$ and $Z$ are mono-indexed. 

Some vertices of $X$ are adjacent to some vertices in $Y$ or some vertices in $Z$ or some vertices in both. The vertices in $X$, which are adjacent to the vertices in $Y$ alone, can be labeled by distinct non-singleton sets which have not already been used for labeling the vertices in $G$. The vertices of $X$, that are adjacent to the vertices in $Z$ alone, can be labeled by distinct singleton sets that are not used before for labeling the vertices in $G$. We note that no edge labeled so far is a mono-indexed edge. 

Assume that a vertex, say $v$ of $X$ is adjacent to a vertex, say $u$, in $Y$ and to a vertex, say $w$ in $Z$. Clearly, $uvw$ is a path of length $2$ with $v$ as an internal vertex. Since $w$ is not mono-indexed, $v$ can only be labeled by a singleton set, not already used for labeling any vertex in $G$. Hence, the edge $uv$ is mono-indexed. Therefore, each path of length $2$ with its internal vertex in $X$ has a mono-indexed edge. This completes the proof.  
\end{proof}

\begin{proposition}\label{T-SNKBSG1}
Let $G$ be a complete bisplit graph and let $X, Y, Z$  be the three partitions of $V(G)$. The sparing number of $G$ is the product of the cardinalities of two of these sets with minimum cardinality.
\end{proposition}
\begin{proof}
Without loss of generality, assume that $|X|\le |Y|\le |Z|$. Then, label all the vertices of $Z$ by distinct non-singleton sets  and label the vertices of $Y$ by distinct singleton sets. No edges between $Y$ and $Z$ are mono-indexed. Since $G$ is a complete bisplit graph, every vertex of $X$ must be adjacent to all the vertices of $Y$ and $Z$. Hence, every vertex of $X$ must be mono-indexed. Hence, all edges between $X$ and $Y$ are mono-indexed and no edges between $X$ and $Z$ are mono-indexed. Therefore, the number of mono-indexed edges in $G$ is the number of edges between $X$ and $Y$. Since $X$ and $Y$ are independent sets, the maximum number of edges between $X$ and $Y$ is $|X|~|Y|$. Therefore, $\varphi(G)=|X|\,|Y|$.
\end{proof}

Note that a complete bisplit graph is a complete tripartite graph. Hence, we rewrite Theorem \ref{T-SNKBSG1} as

\begin{theorem}
The sparing number of a complete tripartite graph is the product of the cardinalities of the two sets having minimum cardinality in its tripartition.
\end{theorem}

\noindent Let us next consider the class of block graphs, defined as follows.

\begin{definition}{\rm
\cite{FH} A graph is called a {\em block graph} or a {\em clique tree} if it is connected and every block is a clique. A graph is a block graph if it can be constructed from a tree by replacing every edge by a clique of arbitrary size, with at most one vertex in common.}
\end{definition}

\noindent The following result discusses about the sparing number of the block graphs.

\begin{theorem}
Let $G$ be a block graph. Then, $G$ admits a weak IASI and its sparing number $G$ is $\frac{1}{2}\sum_{i=1}^r(n_i-1)(n_i-2)$, where $r$ is the number of cliques in $G$ and $n_i$ is the order of the $i$-th clique in $G$.
\end{theorem}
\begin{proof}
A block graph is an edge disjoint union of cliques, any two which have at most one vertex in common. Let $K_{n_1},K_{n_2},K_{n_3},\ldots, K_{n_r}$ be the edge disjoint cliques in $G$. If $K_{n_i}$ and $K_{n_j}$ have a common vertex in $G$ and $K_{n_j}$ and $K_{n_k}$ has another common vertex in $G$, then $K_{n_i}$ and $K_{n_k}$ do not have any common vertex in $G$. Hence, the total number of mono-indexed edges in $G$ is the sum of mono-indexed edges in each clique $K_{n_i}$. By Theorem \ref{T-WKN}, the number of mono-indexed edges in a clique $K_{n_i}$ in $G$ is $\frac{1}{2}(n_i-1)(n_i-2)$. Since $G$ is a graph having edge disjoint cliques, the number of mono-indexed edges in $G$ is the sum of the mono-indexed edges in each clique.  Hence, the total number of mono-indexed edges in $G$ is $\frac{1}{2}\sum_{i=1}^r(n_i-1)(n_i-2)$.
\end{proof}

\noindent Now, consider the following notions.

\begin{definition}{\rm
\cite{BLS} A {\em windmill graph}, denoted by $W(n,r)$, is an undirected graph constructed for $n\ge 2$ and $r\ge 2$ by joining $r$ copies of the complete graph $K_n$ at a shared vertex.}
\end{definition}

\begin{definition}{\rm
\cite{BLS} A {\em friendship graph} or a {\em dutch wind mill} or a {\em fan graph}, denoted by $F_r$, is a graph obtained by joining $r$ copies of the cycle graph $C_3$ with a common vertex. The fan graph $F_r$ is isomorphic to the windmill graph $W(3,r)$.}
\end{definition}

\noindent The following theorem is on the sparing number of windmill graphs.

\begin{theorem}\label{T-WNRG}
A windmill graph $W(n,r)$ admits a weak IASI and the sparing number of $W(n,r)$ is $\frac{r}{2}(n-1)(n-2)$.
\end{theorem}
\begin{proof}
For $1\le i \le r$, let $K_{n_i}$ be the $i$-th copy of $K_n$ in the windmill graph $W(n,r)$, which admits a weak IASI. Conventionally, we assign the IASI $f_i=i.f$ to the copy $K_{n_i}$, where $f$ is an IASI defined on $K_n$. Since any pair of $K_{n_i}$  in $G$ are edge disjoint and has the same vertex in common, the number of mono-indexed edges in $W(n,r)$ is the sum of mono-indexed edges in each $K_{n_i}$.
By Theorem \ref{T-WKN}, each $K_{n_i}$ contains $\frac{(n-1)(n-2)}{2}$ mono-indexed edges. Then, the total number of mono-indexed edges in $G$ is $\displaystyle{\sum_{i=1}^r}\frac{(n-1)(n-2)}{2}=\frac{r}{2}(n-1)(n-2)$. Hence $\varphi(W(n,r))= \frac{r}{2}(n-1)(n-2)$. 
\end{proof}

\begin{corollary}
A friendship graph $F_r= W(3,r)$ admits a weak IASI and the sparing number of $F_r$ is $r$.
\end{corollary}
\begin{proof}
Put $n=3$ in Theorem \ref{T-WNRG}. Then, we have the total number of mono-indexed edges in $F_r$ is $\frac{r}{2}(n-1)(n-2)=\frac{r}{2}.2.1=r$.
\end{proof}

\begin{definition}{\rm
\cite{CZ} The  {\em shadow graph} of a graph $G$ is obtained from $G$ by adding, for each vertex $v$ of $G$, a new vertex $v'$, called the {\em shadow vertex} of $v$, and joining $v'$ to the neighbours of $v$ in $G$. The shadow graph of a graph $G$ is denoted by $S(G)$.}
\end{definition} 

The following theorem establishes the admissibility of weak IASI by the shadow graph $S(G)$ of a weak IASI graph $G$ and finds out the sparing number of $S(G)$.

\begin{theorem}
The shadow graph of a weak IASI graph also admits a weak IASI. Moreover, $\varphi(S(G))=2\,\varphi(G)$.
\end{theorem}
\begin{proof}
Let $v$ be an arbitrary vertex of $G$. Let $v'$ be the shadow vertex of $G$. If $v$ is adjacent to a single vertex, say $u$, then $v'$ is also adjacent to $u$ in $S(G)$. Hence, $vuv'$ is a path of length $2$ in $S(G)$. If $G$ is a tree, then $S(G)$ is also tree and by Theorem \ref{T-WUOC1}, $S(G)$ need not have a mono-indexed edge. 

If the vertex $v$ is adjacent to two vertices, say $u$ and $w$ in $G$, then $v'$ is adjacent to $u$ and $w$ in $S(G)$ forming a cycle $C:vuv'wv$ of length $4$. Then, by Theorem \ref{T-MIEOE}, it has even number of mono-indexed edges. That is, if either $vu$ or $vw$ is mono-indexed in $G$, then the corresponding edge $v'u$ or $v'w$ is also mono-indexed in $S(G)$. Since $v$ is arbitrary, this is true for any vertex of $G$. Hence, $\varphi(S(G))=2\, \varphi(G)$. 
\end{proof}

In the following results, we discuss the admissibility of a weak IASI by some other associated graphs of a given IASI graph $G$.

\begin{definition}{\rm
\cite{BM1} The {\em subdivision} of some edge $e=uv$ is a graph obtained by introducing one new vertex $w$ to the edge $uv$, and with two edges replacing $e$ by two new edges, $uw$ and $wv$. A {\em subdivision} of a graph $G$ is a graph resulting from the subdivision of edges in $G$. 

Two graphs $G$ and $G'$ are said to be {\em homeomorphic} if there is an graph isomorphism from some subdivision of $G$ to some subdivision of $G'$.}
\end{definition}

The following theorem verifies the admissibility of weak IASI by a homeomorphic graph $G'$of a weak IASI graph $G$. 

\begin{theorem}\label{T-WIHG}
A graph $G'$ obtained by the subdivision of an edge $e$ of a weak IASI graph $G$ admits a (induced) weak IASI if and only if $e$ is mono-indexed.
\end{theorem}
\begin{proof}
First observe that if one new element is introduced to an IASI graph in place of another element of $G$, it is customary to assign the same set-label of the replaced element to the newly introduced element. 

Now, assume that the graph $G'$ obtained by subdividing an edge $e=uv$ of a weak IASI graph $G$. Let $w$ be the new vertex introduced to the edge $e$. Then, the set-label of $w$ in $G'$ is the same set-label of $e$ in $G$. If $e$ is not mono-indexed in $G$, then either $u$ or $v$ has a non-singleton set-label. Without loss of generality, let $v$ be the vertex which is not mono-indexed. Then, for the edge $wv$ of $G'$ both the end vertices $u$ and $v$ have non-singleton set-labels, which is a contradiction to the hypothesis that $G'$ admits a weak IASI.

Let $e$ be a mono-indexed edge in $G$. Then, the end vertices $u$ and $v$ of $e$ are mono-indexed in $G$. Let $G'$ be the graph obtained from $G$ by introducing a vertex $w$ to the edge $e$. In $G'$, the vertex $w$ has the same set-label of $e$ in $G$. Hence, the edge $uw$ and $vw$ in $G'$ are mono-indexed in $G'$. Hence, $G'$ is a weak IASI graph. 
\end{proof}

\begin{remark}\label{R-WIHG}{\rm 
Due to Theorem \ref{T-WIHG}, it can be observed that the graph $G'$ obtained by subdividing the mono-indexed edges of a weak IASI graph is also a weak IASI graph. The graph subdivision which is obtained by subdividing all the mono-indexed edges of $G$ is called a {\em maximal subdivision} of $G$ with respect to the weak IASI of $G$.}
\end{remark} 

The following result is on the sparing number of a maximal subdivision $G'$ of a weak IASI graph $G$.

\begin{theorem}
The sparing number of the maximal subdivision graph $G'$ of a weak IASI graph $G$ is $2\, \varphi(G)$. 
\end{theorem}
\begin{proof}
If a vertex $w$ is introduced to the edge $uv$ of $G$ to get a subdivision $G'$, then by Theorem \ref{T-WIHG}, $uv$ is mono-indexed in $G$. Also, in $G'$ the new edges $uw$ and $wv$ are mono-indexed. That is, corresponding to every mono-indexed edge in $G$, there exist two mono-indexed edges in its subdivision graph. Therefore, for a maximal subdivision $G'$ of $G$, $\varphi(G')= 2\,\varphi(G)$.
\end{proof}

\noindent Now, recall the following definition.

\begin{definition}{\rm
\cite{FH} A {\em cactus} or a {\em cactus tree} is a connected graph $G$ in which any two simple cycles have at most one vertex in common. A graph $G$ is a cactus if and only if every block in $G$ is either a simple cycle or a single edge. The following result and estimates sparing number of a cactus graph.}
\end{definition}

\noindent The following theorem estimates the sparing number of a cactus.

\begin{theorem}
The sparing number of a cactus $G$ is $r$, where $r$ is the number of odd cycles in $G$.
\end{theorem}
\begin{proof}

Let $G$ be a cactus. Then, $G$ may be considered as a union of edge disjoint paths (or trees) and cycles, any two of which has at most one vertex in common. Let $C_n$ be a cycle which admits a weak IASI and let $P_m$ be a path (or $T$ be a tree), which has at most one vertex in common. Then, by Theorem \ref{T-WIUG}, $C_n\cup P_m$ (or $C_m\cup T$) admits a weak IASI and $\varphi(C_n\cup P_m) = \varphi(C_m)$ (or $\varphi(C_n\cup T) = \varphi(C_m)$). Hence, in order to estimate the sparing number of a cactus, we need to take the sum of the sparing numbers of all cycles in it.

Let $C_{m_1},C_{m_2},C_{m_3},\ldots, C_{m_n}$ be the edge disjoint cycles in $G$. Without loss of generality, let $C_{m_1},C_{m_2},C_{m_3},\ldots, C_{m_r}$ be the odd cycles and $C_{m_{r+1}}, C_{m_{r+2}},\ldots, C_{m_{n}}$ be the even cycles in $G$.

Now, assume that $G$ admits a weak IASI, each cycle in $G$ must have a weak IASI, say $f_i$, where $f_i$ is the restriction of $f$ to the cycle $C_{m_i}$.  By Theorem \ref{T-WUOC1}, even cycles do not necessarily have a  mono-indexed edge and each odd cycles $C_i$, must have at least one mono-indexed edge. Therefore, $G$ has at least $r$ mono-indexed edges. Hence, $\varphi(G)=r$.
\end{proof}

Our next discussion is on the sparing number of a graph structure called $(m,n)$-cone. First, recall the definition of an $(m,n)$-cone.

\begin{definition}{\rm 
An {\em $(m,n)$-cone} or an {\em $n$-point suspension}, denoted by $C_{m,n}$ is a graph $G$ whose vertex is partitioned into two sets $U$ and $S$ where $S$ is an independent set and the vertices of $U$ forms a cycle in $G$ such that every vertex of $W$ is adjacent with all vertices of $U$. 

For positive integers $m,n$, if $\langle U \rangle=C_m$, then $C_{m,n}=C_m + S$, where $n=|S|$.}
\end{definition}

If $n=1$, then $C_{m,n}$ reduces to a wheel graph $W_{m+1}$ and, the following result discusses the sparing number of a wheel graph $W_{m+1}$. 

\begin{proposition}
\cite{GS4} The sparing number of a wheel graph $W_{m+1}$ is $\lceil \frac{(m-1)}{2} \rceil$.
\end{proposition}

Hence, we estimate the sparing number of an $(m,n)$-cone in the following theorem. 

\begin{theorem}\label{T-SNMNC}
For $m>0, n\ge 2$, the sparing number of an $(m,n)$-cone is $m$.
\end{theorem}
\begin{proof}
For $m>0, n\ge2$, let $C_{m,n}= C_m+S$, where $S$ be an independent set in $G$. Now we have the following cases.

\noindent {\em Case-1:}~  Let $m$ be an even integer. Hence, $C_m$ is an even cycle and if we label the vertices of $C_m$ first, then by Theorem \ref{T-WUOC1}, $\varphi(C_m)=0$ and it has exactly $\frac{m}{2}$ mono-indexed vertices and exactly $\frac{m}{2}$ vertices that are not mono-indexed. Since each vertex $w_j$ of $S$ is adjacent to all vertices of $C_m$, all vertices of $S$ must be mono-indexed. Therefore, exactly $\frac{m}{2}$ edges incident on each vertex $w_j$ of $S$ are mono-indexed. The number of mono-indexed edges between $C_m$ and $S$ is $\frac{mn}{2}$. Hence, the total number of mono-indexed edges in $C_{m,n}$ in this case is $\frac{mn}{2}$.

If the vertices of $S$ are labeled first, then all vertices of $S$ can be labeled by distinct non-singleton sets, since $S$ is independent set in $C_{m,n}$. Then, every vertex of $C_m$ must be mono-indexed. In this case, no edge between $C_m$ and $S$ is mono-indexed and the total number of mono-indexed edges in $C_{m,n}$ is $m$.

Since $n\ge 2$, we have $\frac{mn}{2}>m$.

\noindent {\em Case-2:}~  Let $m$ be an odd integer. Hence, $C_m$ is an odd cycle and if we label the vertices of $C_m$ first, then by Theorem \ref{T-WUOC}, $\varphi(C_m)=1$ and it has exactly $\frac{(m+1)}{2}$ mono-indexed vertices and exactly $\frac{(m-1)}{2}$ vertices that are not mono-indexed. Since each vertex $w_j$ of $S$ is adjacent to all vertices of $C_m$, all vertices of $S$ must be mono-indexed. Therefore, exactly $\frac{(m+1)}{2}$ edges incident on each vertex $w_j$ of $S$ are mono-indexed. The number of mono-indexed edges between $C_m$ and $S$ is $\frac{(m+1)n}{2}$. Hence, the total number of mono-indexed edges in $C_{m,n}$ in this case is $1+\frac{(m+1)n}{2}$.

If the vertices of $S$ are labeled first, then as explained in Case-1, all vertices of $S$ can be labeled by distinct non-singleton sets. Then, every vertex of $C_m$ must be mono-indexed. In this case no edge between $C_m$ and $S$ is mono-indexed and the total number of mono-indexed edges in $C_{m,n}$ is $m$.

Since $n\ge 2$, we have $1+\frac{(m+1)n}{2}>m$.

In both cases, the minimum number of mono-indexed edges in $C_{m,n}$ is $m$. Hence, the sparing number $\varphi(C_{m,n})=m$. 
\end{proof}
 
\section{Conclusion}

In this paper, we have established some results on the sparing number of certain graphs and graph classes. The admissibility of weak IASI by certain other graph classes, graph operations and graph products and finding the corresponding sparing numbers are still open. There are several other open problems regarding the necessary and sufficient conditions for the admissibility of certain IASIs, both uniform and non-uniform, by various graphs and graph classes.

\end{document}